\documentclass[12pt,twoside,reqno]{amsart}
\usepackage[letterpaper,textwidth=6.9in,textheight=9.6in,centering]{geometry}
\usepackage{amsmath, amsfonts, amssymb}
\usepackage{mathrsfs}
\usepackage{latexsym}
\usepackage{mathtools}
\usepackage{exscale}
\usepackage{cases}
\usepackage{enumitem}
\usepackage{needspace}

\usepackage[pdfstartview=FitH,
            CJKbookmarks=true,
            bookmarksnumbered=true,
            bookmarksopen=true,
            hypertexnames=false,
            colorlinks=true,
            linkcolor=blue,
            anchorcolor=blue,
            citecolor=red,
            urlcolor=blue
            ]{hyperref}
\hypersetup{
    pdftitle={On the Reverse Projection Inequality},
    pdfauthor={},
    pdfsubject={The sharp three-dimensional inequality, its equality cases, and geometric high-dimensional counterexamples},
    pdfkeywords={projection body, reverse projection inequality, mixed volume, equality case, simplex section, cube projection, geometric perturbation}
}

\numberwithin{equation}{section}
\allowdisplaybreaks

\newtheorem{theorem}{Theorem}[section]

\newtheorem{lemma}[theorem]{Lemma}
\newtheorem{corollary}[theorem]{Corollary}

\newtheorem{proposition}[theorem]{Proposition}

\newtheorem{definition and theorem}[theorem]{Definition and Theorem}

\def\bl{\begin{lemma}}
\def\el{\end{lemma}}
\def\bc{\begin{corollary}}
\def\ec{\end{corollary}}
\def\bt{\begin{theorem}}
\def\et{\end{theorem}}

\def\bp{\begin{proposition}}
\def\ep{\end{proposition}}
\def\be{\begin{equation}}
\def\ee{\end{equation}}
\def\baa{\begin{align*}}
\def\eaa{\end{align*}}

\theoremstyle{definition}
\theoremstyle{remark}

\DeclareMathOperator{\conv}{conv}

\DeclareMathOperator{\lin}{lin}

\DeclareMathOperator{\ext}{ext}
\DeclareMathOperator{\pos}{pos}

\DeclareMathOperator{\relint}{relint}
\DeclareMathOperator{\supp}{supp}

\newcommand{\R}{\mathbb{R}}

\newcommand{\Sph}{\mathbb S}
\newcommand{\one}{\mathbf 1}

\title[On the Reverse Projection Inequality]
{\large On the Reverse Projection Inequality}

\author[Y. Feng]{Yibin Feng}
\address{Yibin Feng, School of Mathematics and Statistics,
	Lanzhou University,
	Lanzhou, 730000, P.R. China.}
\email{fengyb@lzu.edu.cn}

\author[S. Hu]
{Shengnan Hu}
\address
{Shengnan Hu, MOE-LCSM, School of Mathematics and Statistics, Hunan Normal University, Changsha, 410081, P.R. China.}
\email{helen@hunnu.edu.cn}

\author[W. Liu]
{Weiru Liu}
\address
{Weiru Liu, School of Mathematics and Statistics, Central China Normal University, 
	Wuhan, 430079, P.R. China.}
\email{lwr1997@ccnu.edu.cn}

\author[L. Xu]
{Lei Xu}
\address
{Lei Xu, School of Statistics and Mathematics, Zhongnan University of Economics and Law, Wuhan, 430073, P.R. China.}
\email{leixumath@zuel.edu.cn}


\begin{document}
	
	\begin{abstract}
		In \cite {ChenFengLiXiXu2026}, a reverse projection inequality for general convex bodies was established in three dimensions. In this paper, we completely characterize the equality cases of this inequality. We further show that the reverse projection inequality fails in every dimension $n\geq9$. The counterexamples are constructed as full-dimensional polytopes with at most $n+2$ facets.
	\end{abstract}
	
	\maketitle
	
	\section{Introduction and main results}
	\enlargethispage{2pt}
	
	The setting of this paper is the $n$-dimensional Euclidean space $\mathbb{R}^n$. The unit sphere in $\mathbb{R}^n$ is denoted by $\mathbb{S}^{n-1}$. A convex body in $\mathbb{R}^n$ is a compact convex set that has a non-empty interior. Let $\mathcal{K}^n$ denote the set of all convex bodies in $\mathbb{R}^n$. We write $\mathcal{K}_o^n$ for the class of convex bodies containing the origin in their interiors and $\mathcal{K}_e^n$ for the class of origin-symmetric convex bodies. We use $\mathcal{H}^k$ to denote the $k$-dimensional Hausdorff measure.
	
For $K \in \mathcal{K}^n$, its projection body $\Pi K$ is the
	origin-symmetric convex body with support function
	\begin{equation*}\label{eq:projection-body-definition}
		h_{\Pi K}(u)=V_{n-1}(K\mid u^\perp),\qquad u\in\Sph^{n-1}.
	\end{equation*}
	Here $V_j$ denotes $j$-dimensional volume and $K\mid u^\perp$ is the
	orthogonal projection of $K$ onto $u^\perp$; see
	\cite[Chapter~3]{Gardner2006} and \cite[Section~9.1]{Schneider2014}.  The projection body operator satisfies
	\begin{equation*}\label{eq:affine-transformation}
		\Pi(\phi K)=|\det\phi|\,\phi^{-t}\Pi K,
		\qquad \phi\in\mathrm{GL}(n).
	\end{equation*}
It follows immediately that 
	\begin{equation}\label{eq:def-Pn}
		\mathcal P_n(K):=V_n(\Pi K)/V_n(K)^{n-1}
	\end{equation}
	is invariant under nonsingular affine transformations. The affine invariant gives rise to two extremal problems. 
	
	For $n\geq3$, the minimization problem defined in \eqref{eq:def-Pn} is Petty's conjectured
	projection inequality
	\begin{equation}\label{eq:petty-lower}
		\mathcal P_n(K)\geq
		\frac{\omega_{n-1}^{\,n}}{\omega_n^{\,n-2}},
	\end{equation}
	where $\omega_j$ is the volume of the Euclidean unit ball in $\R^j$, and
	ellipsoids are the conjectured equality cases; see
	\cite{Petty1967,Petty1972,Lutwak1990}. In $\mathbb{R}^3$, the constant in
	\eqref{eq:petty-lower} is $3\pi^2/4$, and the sharp inequality, including its equality cases, was recently established in
	\cite[Theorem~1.1]{ChenFengLiXiXu2026}.  
	
	The present paper is concerned with the maximization problem for the affine invariant defined in \eqref{eq:def-Pn}. This problem arises from Schneider's study of random hyperplanes meeting a convex body \cite{Schneider1982}.  In its original setting of origin-symmetric convex bodies, the problem, commonly known as Schneider’s projection problem, asks for the optimal upper bound of $\mathcal P_n(K)^{1/n}$. Schneider
	conjectured that $\mathcal P_n(K)\leq2^n$ for every origin-symmetric convex body $K$, with equality if and only if $K$ is a Weil body (an affine images of a Cartesian products whose factors are either line segments or centrally symmetric planar convex bodies).
	Brannen \cite{Brannen1996} disproved this conjecture by constructing a counterexample.
	
	This paper focuses on the inequality by Brannen \cite{Brannen1996} for arbitrary convex bodies. That is, for every convex body  $K \in \mathcal{K}^n$,
	\begin{equation}\label{eq:brannen-general}
		\mathcal P_n(K)\leq \frac{(n+1)n^n}{n!}=\mathcal P_n(Z),
	\end{equation}
	with equality only for simplices $Z$. Several important estimates related to Schneider’s projection problem were obtained in
	\cite{BourgainLindenstrauss1988,GiannopoulosPapadimitrakis1999,LutwakYangZhang2001}.
	In particular, Lutwak-Yang-Zhang \cite{LutwakYangZhang2001} established the following non-sharp upper bounds:
	\[
	\mathcal P_n(K)\leq
	\begin{cases}
	2^n\bigl(n^n/n!\bigr)^{1/2},&\text {for }K \in \mathcal{K}_e^n,\\[2mm]
	n^n(n+1)^{(n+1)/2}(n!)^{-3/2},&\text{for }K \in \mathcal{K}^n,
	\end{cases}
	\]
	which are asymptotically optimal on the exponential scale. They also
	introduced a related centro-affine functional and established a modified
	version of Schneider’s projection problem for which parallelotopes are the equality
	cases \cite{LutwakYangZhang2001}. Brannen \cite{Brannen2005} investigated certain classes of three-dimensional projection bodies. Later, Saroglou \cite{Saroglou2011} determined the sharp maximum of $\mathcal P_3$ within the classes of three-dimensional zonoids, cones, and double cones, and characterized the corresponding equality cases; see
	also \cite{Saroglou2015} for the role of the second projection body and
	class-reduction arguments. Henk \cite{Henk2025} proved that every
	$n$-dimensional zonotope $Z$ with exactly $n+1$ generators satisfies
	$\mathcal P_n(Z)=2^n$, and constructed origin-symmetric examples showing
	that the symmetric supremum is at least
	$2^n(9/8)^{\lfloor n/3\rfloor}$ where $\lfloor n/3\rfloor$ is the greatest integer not exceeding $n/3$. 
	
	The sharp upper bound for $\mathcal P_3$ over the class $\mathcal{K}^3$ was established in \cite[Theorem~1.2]{ChenFengLiXiXu2026}. We recall the result below.
	
\begin{theorem}\cite[Theorem~1.2]{ChenFengLiXiXu2026} \label{thm:upper-inequality}
	For every convex body  $K \in \mathcal{K}^3$,
	\begin{equation}\label{eq:main-inequality}
	V_3(\Pi K)\leq18\,V_3(K)^2.
	\end{equation}
	The constant is sharp, and every tetrahedron attains equality.
\end{theorem}

	Our first result shows that equality in \eqref{eq:main-inequality} holds only for tetrahedron.
	
	\begin{theorem}\label{thm:equality-case}
		Let $K \in \mathcal{K}^3$. If
		\[
		V_3(\Pi K)=18V_3(K)^2,
		\]
		then $K$ is a tetrahedron.
	\end{theorem}

	Our second result shows that inequality \eqref{eq:brannen-general} fails in every dimension $n\geq9$.

	\begin{theorem}
		\label{thm:high-dimensional-counterexamples}
		For every integer $n\geq9$, there exists a full-dimensional polytope
		$Q_n\subset\R^n$ with at most $n+2$ facets, such that
		\begin{equation}\label{eq:main-lower}
			\mathcal P_n(Q_n)>
			\frac{(n+1)n^n}{n!}.
		\end{equation}
	\end{theorem}

	The known results reviewed above, together with our findings, show that whether inequality \eqref{eq:brannen-general} holds depends on the dimension. The simplex bound in inequality \eqref{eq:brannen-general} is valid and sharp in dimensions two and three, and
	Theorem~\ref{thm:equality-case} completes the equality classification in
	dimension three.  On the other hand, Theorem~\ref{thm:high-dimensional-counterexamples} shows that the bound fails in every dimension  $n\geq9$. The remaining cases $4\leq n\leq8$ are still unresolved.

\noindent\textit{Plan of the paper.}
Section~\ref{a} introduces the notation and preliminary results used throughout the paper and presents the equality reduction derived from the known three-dimensional upper bound. Section~\ref{b} establishes the complete equality
characterization in dimension three by reducing the equality cases to
generalized pyramids and computing their projection volumes.
Section~\ref{c} constructs
a counterexample in dimension nine through a geometric first-variation
argument, and then extends the construction to every dimension
\(n\geq 9\) by means of a lifting procedure.
	
	\section{Preliminaries}\label{a}
	
	In this section, we collect some basic facts and notation concerning convex bodies. For further background, we refer the reader to the monographs by Schneider \cite{Schneider2014} and Gardner \cite{Gardner2006} and the paper \cite{ChenFengLiXiXu2026}.
	
	Let \(K\in \mathcal{K}_o^n\). The polar body of \(K\) is defined by
	\[
	K^\circ
	=
	\{
		y\in\mathbb{R}^n:
		\langle x,y\rangle\leq 1
		\text{ for every }x\in K
		\}.
	\]
	The correspondence \(K\mapsto K^\circ\) is called polarity. It reverses inclusions and satisfies
	\[
	(K^\circ)^\circ=K.
	\]
	When \(K\) is a polytope, polarity induces an inclusion-reversing correspondence between the nonempty proper faces of \(K\) and those of \(K^\circ\). More precisely, a \(k\)-dimensional face of \(K\) corresponds to an \((n-k-1)\)-dimensional face of \(K^\circ\). In particular, in dimension three, facets of \(K\) correspond to vertices of \(K^\circ\), edges correspond to edges, and vertices correspond to facets.
	
	For $u_1,u_2,u_3\in\Sph^2$, write
	\[
	\operatorname{pos}\{u_1,u_2,u_3\}
	:=\{\lambda_1u_1+\lambda_2u_2+\lambda_3u_3:
		\lambda_1, \lambda_2,\lambda_3\geq0\}.
	\]
	This is their positive hull, equivalently the smallest convex cone containing $u_1,u_2,u_3$.
	We define the {\it normalized solid angle} of this cone by
	\[
	\Gamma(u_1,u_2,u_3)
	=\frac{\mathcal H^2\bigl(\operatorname{pos}\{u_1,u_2,u_3\}\cap\Sph^2\bigr)}{4\pi}.
	\]
Then, Schneider’s mixed translative functional is given by
\begin{equation*}\label{translativefunctional}
	\mathcal B(K_1,K_2,K_3)
	=
	\int_{\Sph^2\times \Sph^2\times \Sph^2}\Gamma(u_1,u_2,u_3)
	\left| \det(u_1,u_2,u_3)\right|
	\prod_{i=1}^3  d S_2(K_i,u_i),
\end{equation*}
where $K_1,K_2,K_3\in \mathcal{K}^3$. The functional $\mathcal B$ is symmetric and invariant under simultaneous
reflection of all three arguments. Hence, for $K\in \mathcal{K}^3$ we have
\[
\mathcal B(-K,-K,-K)=\mathcal B(K,K,K),\qquad
\mathcal B(K,-K,-K)=\mathcal B(K,K,-K).
\]

Let \(L_1,\ldots,L_r\subset\mathbb{R}^m\) be compact convex sets.
Their mixed volumes are defined by the Minkowski polynomial identity
\[
{V}_m\left(
\lambda_1L_1+\cdots+\lambda_rL_r
\right)
=
\sum_{i_1,\ldots,i_m=1}^{r}
\lambda_{i_1}\cdots\lambda_{i_m}
V_m\left(L_{i_1},\ldots,L_{i_m}\right),
\ \ \
\lambda_1,\ldots,\lambda_r\geq 0.
\]
We write \(L_j[\alpha_j]\) when \(L_j\) occurs
\(\alpha_j\) times among the arguments. In particular,
\[
V_m(L[m])={V}_m(L).
\]
In $\R^6=\R^3\times\R^3$. we define that for $K\in \mathcal{K}^3$,
		\[
	v=V_3(K),\ \ 
	b=90V_6\bigl((-\bigtriangleup_2K)[2],\iota_1K[2],\iota_2K[2]\bigr),\ \ 
	c=90V_6\bigl(\bigtriangleup_2K[2],\iota_1K[2],\iota_2K[2]\bigr).
	\]
Here,
	\[
\bigtriangleup_2K=\{(x,x):x\in K\},\ \ 
\iota_1K=K\times\{0\},\ \ 
\iota_2K=\{0\}\times K.
\]
By
	\cite[Lemma~5.1]{ChenFengLiXiXu2026}, we see that
	\[
		b=\mathcal B(-K,-K,-K),\ \
		c=\mathcal B(K,-K,-K).
	\]
Equations (5.5), (5.6) and (5.9) in \cite{ChenFengLiXiXu2026} give
	\begin{equation*}\label{eq:quoted-upper-data}
		V_3(\Pi K)=\frac b3+c,\ \ \
		b\leq27v^2,\ \ \  c\leq9v^2.
	\end{equation*}
This implies	
\begin{equation*}\label{eq:quoted-upper-data1}
	V_3(\Pi K)\leq 18v^2.
\end{equation*}		
Since 		
\begin{equation*}
	18v^2-V_3(\Pi K)=\frac{27v^2-b}{3}+(9v^2-c),
\end{equation*}				
noting that $27v^2-b\geq0$ and $9v^2-c\geq0$, we obtain
	\begin{equation}\label{eq:both-equalities}
		V_3(\Pi K)=18v^2
		\quad\Longleftrightarrow\quad
		b=27v^2\ \text{ and }\ c=9v^2.
	\end{equation}

	\section{Proof of Theorem~\ref{thm:equality-case}}\label{b}

	A finite multiset of compact convex sets $L_1,\ldots,L_r\subset\R^n$ is called
	\emph{essential} if
	\[
	\dim\Bigl(\sum_{i\in I}L_i\Bigr)\geq |I|
	\]
	for every nonempty index set $I\subseteq\{1,\ldots,r\}$.  Here, 
	\[
	\sum_{i\in I}L_i=\left\{
	\sum_{i\in I} x_i : x_i\in L_i
	\right\},
	\]
and $|I|$ denotes the number of selected indices. 
	
	We first characterize the equality case $b=27v^2$ for polytopes $K$.  
	
	\begin{proposition}\label{prop:polytope-b-equality}
		Let $K\subset\R^3$ be a convex polytope.  Then
		\begin{equation}\label{eq:b-facet-condition}
			b=27V_3(K)^2
			\quad\Longleftrightarrow\quad
			\text{every three distinct facets of $K$ have a common point}.
		\end{equation}
		In particular, if $V_3(\Pi K)=18V_3(K)^2$ then $K$ is a tetrahedron.
	\end{proposition}
	
	\begin{proof}
		
		Let $K\in \mathcal{K}^3$. In $\R^3\times\R^3\times\R^3$, we write
		\[
		\begin{gathered}
		D =\{(x,x,x):x\in K\},\\
		J_0=K\times\{0\}\times\{0\},\\
		J_1=\{0\}\times K\times\{0\},\\
		J_2=\{0\}\times\{0\}\times K,\\
	Q=J_0+J_1+J_2=K\times K\times K,
		\end{gathered}
		\]
where the addition is Minkowski addition. Define
		\[
		\begin{aligned}
		M&=V_9( D[3],J_0[2],J_1[2],J_2[2]),\\
		N&=V_9( Q[3],J_0[2],J_1[2],J_2[2]).
		\end{aligned}
		\]
		
		The calculation leading to equations~(5.7) and (5.8) in \cite {ChenFengLiXiXu2026} and the derivation of (5.9) in \cite {ChenFengLiXiXu2026} therein yield
		\[
		M=\frac{vb}{7560},\qquad M\leq N,\qquad N=\frac{v^3}{280}.
		\]
	Therefore,
		\[
		27v^2-b=\frac{7560}{v}(N-M).
		\]
		Thus
		\begin{equation}\label{eq:b-MN-equality}
			b=27v^2\quad\Longleftrightarrow\quad M=N.
		\end{equation}
	Consider the ordered nine-tuples
		\[
		\begin{split}
			\mathscr P=( D, D, D,
			J_0,J_0,J_1,J_1,J_2,J_2),\ \ \ \ 
			\mathscr Q=(Q, Q,Q,
			J_0,J_0,J_1,J_1,J_2,J_2).
		\end{split}
		\]
		Let $u=(u_0,u_1,u_2)\in(\R^3\times \R^3\times \R^3)\setminus\{0\}$. Define
		\[
		F_i:=F(K,u_i)=\{x\in K:\langle x,u_i\rangle=h_K(u_i)\},
		\qquad F(K,0):=K.
		\]
		Then
		\[
	 Q^u:=F(Q, u)=F_0\times F_1\times F_2,
		\]
		and 
		\begin{equation}\label{eq:triple-face-intersection}
			D\cap Q^u\neq\varnothing
			\quad\Longleftrightarrow\quad
			F_0\cap F_1\cap F_2\neq\varnothing.
		\end{equation}
		
	Let
		$(P_i)_{i=1}^9$ and $(Q_i)_{i=1}^9$ be the members of $\mathscr P$ and
		$\mathscr Q$, respectively, and define
		$T_u=\{i:P_i\cap Q_i^u\neq\varnothing\}$. The strict-monotonicity criterion of Bihan and Soprunov
		\cite[Theorem~3.3]{BihanSoprunov2019} implies that $M<N$ if and only if, for some $u\neq 0$, the collection
		\[
		\{Q_i^u:i\in T_u\}\cup\{Q_i:i\notin T_u\}
		\]
		is essential. 
		
		If
		\eqref{eq:triple-face-intersection} holds, every member of $\mathscr P$
		touches the exposed face of the corresponding member of $\mathscr Q$.
		The nine exposed faces lie in translates of the common hyperplane $u^\perp$,
		Therefore, $\{Q_i^u:1\leq i\leq9\}$ is not essential. If
		\eqref{eq:triple-face-intersection} fails, the multiset tested by the
		criterion is
		\[
		\{Q, Q, Q,
		\iota_0F_0,\iota_0F_0,
		\iota_1F_1,\iota_1F_1,
		\iota_2F_2,\iota_2F_2\},
		\]
		where $\iota_i$ embeds $\R^3$ into the $i$th coordinate subspace of
		 $\R^3\times \R^3\times \R^3$.  Any subcollection containing a copy of $Q$ has
		nine-dimensional Minkowski sum. For a subcollection containing no copy of $Q$, all its members lie in the three coordinate subspaces, which form a direct sum. Examining the two repeated copies in each coordinate shows that the displayed multiset is essential if and only if
		\[
		\dim F_i\geq2\qquad(i=0,1,2).
		\]
		Thus, the inequality \(M\leq N\) is strict precisely when there exist \(F_0,F_1,F_2\), each of which is either a facet of \(K\) or \(K\) itself, such that
		\[
		F_0\cap F_1\cap F_2=\varnothing,
		\]
		where repetitions among the facets are allowed. If one
		\(F_i\) equals \(K\), then the other two facets must be disjoint.
		If two of the \(F_i\) coincide, then the repeated facet must be
		disjoint from the remaining one. In either case, replacing the
		occurrence of \(K\), or one of the repeated copies, by another
		facet yields three distinct facets with empty intersection.
		Conversely, if three distinct facets of \(K\) have empty intersection, then choosing an outer normal for each of them yields the strict inequality \(M<N\). Therefore,
		\[
		M=N
		\quad\Longleftrightarrow\quad
		\text{every three distinct facets of \(K\) have a common point}.
		\]
		In view of \eqref{eq:b-MN-equality}, this proves \eqref{eq:b-facet-condition}.
			
		To prove the final assertion of the proposition, it remains to show
		that every three-dimensional polytope satisfying this condition is
		a tetrahedron.  Translate \(K\) so that \(0\in\operatorname{int}K\), and let $L=K^\circ$. Under polarity, the facets of \(K\) correspond to the vertices of \(L\), and condition \eqref{eq:b-facet-condition} becomes the assertion that every three vertices of \(L\) lie in a common proper face. Fix a facet \(G\) of \(L\) and choose a vertex \(z\notin G\). For any two distinct vertices \(x,y\) of \(G\), the preceding property yields a proper face \(H\) of \(L\) containing \(x,y,z\). Since \(z\notin G\), whereas \(x,y\in G\), the three vertices \(x,y,z\) are not collinear. Consequently, \(H\) is two-dimensional and hence is a facet of \(L\). Moreover, \(H\neq G\), and \(H\cap G\) is a face containing both \(x\) and \(y\). Since two distinct facets of a three-dimensional polytope intersect in at most an edge, \(H\cap G\) must be an edge. Thus \(x\) and \(y\) are adjacent. It follows that every pair of vertices of \(G\) is adjacent, forcing \(G\) to be a triangle. Since \(G\) was arbitrary, \(L\) is simplicial. Finally, every triple of vertices of \(L\) lies in a proper face, which, by simpliciality, must be the triangular facet spanned by those three vertices. Hence, every triple of vertices of \(L\) forms a facet.
		
		Let \(m=f_0(L)\) denote the number of vertices of \(L\). Since every triple of vertices of \(L\) forms a triangular facet, we have
		\[
		f_2(L)=\binom{m}{3}.
		\]
		Since \(L\) is a simplicial, each facet has three edges, and each edge belongs to exactly two facets. 
		Counting the edge-facet incidences yields
		\[
		3f_2(L)=2f_1(L).
		\]
		Combining this identity with Euler's relation
		\[
		f_0(L)-f_1(L)+f_2(L)=2
		\]
		and using \(f_0(L)=m\), we obtain
		\[
		f_2(L)=2m-4.
		\]
		Consequently,
		\[
		\binom{m}{3}=2m-4.
		\]
		Since every three-dimensional polytope has at least four vertices, we have \(m\geq4\). Hence
		\[
		\frac{m(m-1)(m-2)}{6}=2(m-2),
		\]
		and division by \(m-2>0\) yields
		\[
		m(m-1)=12.
		\]
		Thus \(m=4\). This implies that \(L\) is a tetrahedron. Since the polar of a tetrahedron is again a tetrahedron, \(K\) is also a tetrahedron. The final assertion now follows from \eqref{eq:both-equalities}.
	\end{proof}

	We next characterize the equality case $c=9v^2$ among convex bodies $K \in \mathcal{K}^3$.  
	
	Translate
	$K$ such that $0\in\operatorname{int}K$. Let
	\[
	L=K^\circ,\ \ \ C=\bigtriangleup_2K,\ \ \
	E_1=\iota_1K,\ \ \ E_2=\iota_2K,
	\ \ \ Q=E_1+E_2=K\times K.
	\]
Equation~(5.1) in \cite{ChenFengLiXiXu2026} gives
	\begin{equation}\label{eq:c-factorized-terms}
		\begin{aligned}
		V_6(C,E_1[2],E_2[3])
		&=V_6(C,E_1[3],E_2[2])=\frac{v^2}{20},\\
		V_6(C[2],E_1,E_2[3])
		&=V_6(C[2],E_1[3],E_2)=\frac{v^2}{20}.
		\end{aligned}
	\end{equation}
Indeed, it follows from the complementary-subspace factorization to each term that the three bodies contained in one coordinate subspace have intrinsic mixed volume \(V_3(K)=v\), while the projections of the remaining three bodies onto the complementary coordinate subspace also have mixed volume \(v\). Thus, each mixed volume in \eqref{eq:c-factorized-terms} equals \(v^2/20\)  since the normalization introduces the binomial coefficient \(\binom{6}{3}=20\). 

Expanding each copy of \(Q=E_1+E_2\) by multilinearity and discarding the terms that vanish by dimensionality, it follows from \eqref{eq:c-factorized-terms} and the
	definition of $c$ that
	\begin{align}
		V_6(C,Q[5])
		&=\binom52V_6(C,E_1[2],E_2[3])
		+\binom53V_6(C,E_1[3],E_2[2])\notag\\
		&=v^2,\label{eq:CQ5-explicit}\\
		V_6(C[2],Q[4])
		&=4V_6(C[2],E_1,E_2[3])
		+6V_6(C[2],E_1[2],E_2[2])
		+4V_6(C[2],E_1[3],E_2)\notag\\
		&=\frac{2v^2}{5}+\frac{c}{15}.
		\label{eq:C2Q4-explicit}
	\end{align}
	Since $C\subset Q$, we have
	\begin{equation}\label{eq:six-dimensional-monotonicity}
		V_6(C,C,Q[4])\leq V_6(Q,C,Q[4])=V_6(C,Q[5]).
	\end{equation}
	By subtracting \eqref{eq:C2Q4-explicit} from
	\eqref{eq:CQ5-explicit}, it follows that
	\begin{equation}\label{eq:c-mixed-volume-deficit}
		9v^2-c
		=15\bigl[V_6(Q,C,Q[4])-V_6(C,C,Q[4])\bigr].
	\end{equation}
	Thus, $c=9v^2$ if and only if equality holds in
	\eqref{eq:six-dimensional-monotonicity}.
	
	Let \(H\subset\mathbb{R}^m\) be a nonempty compact convex set and let \(w\neq0\). The touching cone \(T(H,w)\) is defined as the unique cone that is a face of some normal cone of \(H\) and satisfies
	\(
	w\in\operatorname{relint} T(H,w).
	\)
	We use the homogeneous convention
	$T(H,tw)=T(H,w)$ for $t>0$, and define
	\[
	T(H,0)=\lin(H-H)^\perp.
	\]
	A direction \(w\in\mathbb S^5\) is called \((Q[4],C)\)-extreme if every submultiset \(I\) of the multiset
	\(
	\{Q,Q,Q,Q,C\}
	\)
	satisfies
	\[
	\dim\left(\bigcap_{H\in I}T(H,w)\right)\leq 6-|I|,
	\]
	Note that \(|I|\) denotes the cardinality of \(I\), with multiplicities taken into account. Let $S_{Q[4],C}$ be the mixed area measure associated with four copies
	of $Q$ and one copy of $C$. Using the support theorem of Handel and Wang
	\cite[Theorem~1.10]{vanHandelWang2025} for arbitrary nonempty compact
	convex sets, including the lower-dimensional diagonal $C$, it follows that
	\begin{equation}\label{eq:support-theorem}
		\supp S_{Q[4],C}
		=\overline{\{w\in\Sph^5:w\text{ is }(Q[4],C)\text{-extreme}\}}.
	\end{equation}
	
	For $w=(p,q)\in\R^3\times\R^3$, we define
	\begin{equation*}\label{eq:gap}
		g(p,q)=h_Q(p,q)-h_C(p,q)
		=h_K(p)+h_K(q)-h_K(p+q)\geq0.
	\end{equation*}
	The representation of mixed volume yields
	\begin{equation}\label{eq:mixed-area-deficit}
		\begin{split}
			&V_6(Q,C,Q[4])-V_6(C,C,Q[4])\\
			&\hspace{2cm}=\frac16\int_{\Sph^5}g\,dS_{Q[4],C}.
		\end{split}
	\end{equation}
	Thus, equality in \eqref{eq:six-dimensional-monotonicity} holds if and only
	if $g$ vanishes on $\supp S_{Q[4],C}$.
	
		Write $\tau(r)=T(K,r)$ and let
	$s:\R^3\times\R^3\to\R^3$ be $s(a,b)=a+b$. Using \(Q=K\times K\), 
	\(
	C=\{(x,x)\in K\},
	\)
	and the definitions of the normal and touching cones, we obtain
		\begin{equation}\label{eq:touching-product-diagonal}
		\begin{split}
	T(Q,(p,q))=\tau(p)\times\tau(q),
	\ \ \ \
	T(C,(p,q))=s^{-1}\bigl(\tau(p+q)\bigr).
\end{split}
\end{equation}
	Here, \(s^{-1}\) denotes the inverse image under \(s\). Since the multiset associated with the mixed area measure in \eqref{eq:support-theorem} consists of four copies of \(Q\) and one copy of \(C\), the extremality condition reduces to three cases: the four copies of \(Q\), the single copy of \(C\), and all five bodies together. In fact, repeating \(Q\) does not change the corresponding intersection of touching cones, but it decreases the upper bound \(6-|I|\). Therefore, \(w\) is \((Q[4],C)\)-extreme if and only if
\begin{equation}\label{eq:extreme-three-conditions}
	\begin{split}
	\dim T(Q,w)\le 2,\qquad
	\dim T(C,w)\le 5,\qquad
	\dim\bigl(T(Q,w)\cap T(C,w)\bigr)\le 1.
\end{split}
\end{equation}
	
	We denote by \(\operatorname{ext}L\) the set of extreme points of \(L\). A point \(x\in L\) is called an extreme point if
	\[
	x=(1-t)y+tz,\qquad y,z\in L,\quad 0<t<1,
	\]
	implies \(x=y=z\). In particular, when \(L\) is a polytope, \(\operatorname{ext}L\) is precisely its set of vertices. The following proposition characterizes the equality case \(c=9v^2\) in terms of the facial structure of \(L\).
	\begin{proposition}\label{prop:c-equality-facial}
		Let $L=K^\circ$ as above. Then,
		\begin{equation}\label{eq:c-facial-equivalence}
			\begin{aligned}
				c=9v^2
				\ \ \Longleftrightarrow\quad
				\text{every two points of $\ext L$ lie in a common proper exposed face of $L$. }
			\end{aligned}
		\end{equation}
	\end{proposition}
	\begin{proof}
		
		Assume that \(c=9v^2\), and let \(x,y\in\operatorname{ext}L\). If \(x\) and \(y\) lie in no common proper exposed face, then
		\[
		(x,y)\subset\operatorname{int}L.
		\]

		Suppose first that $x$ and $y$ are linearly independent. Take
		$P=\operatorname{span}\{x,y\}$ and $W=L\cap P$. Then every \(z\in P\) can be written uniquely as \(z=\alpha x+\beta y\). Since \(x\) and \(y\) are extreme points of \(L\), the affine line \(\{\alpha+\beta=1\}\) intersects \(W\) precisely in the segment \([x,y]\). That is,
		\[
		W\cap\{\alpha+\beta=1\}=[x,y].
		\]
		Since $(x,y)\subset\relint W$, $W$ has points in
		$\{\alpha+\beta>1\}$.  Choose an extreme point $z$ of the exposed face of
		$W$ on which $\alpha+\beta$ is maximal. Then
		\begin{equation}\label{eq:z-coordinates}
			z=\alpha x+\beta y,\qquad
			\alpha>0,\quad\beta>0,\quad\alpha+\beta>1.
		\end{equation}

		The point $z$ lies on $\partial L$.  Let $F$ be the unique minimal face of
		$L$ whose relative interior contains $z$.  Since $z$ is extreme in the
		planar section $W$, we must have $\dim F\leq1$.  If $F$ is a segment, its linear
		span cannot equal $P$.  Moreover, since \(0\in\operatorname{int}L\) the convex cone generated by \(F\) is pointed. Hence,
		\begin{equation}\label{eq:touching-polar-duality}
			\tau(x)=\R_+x,\qquad \tau(y)=\R_+y,\qquad
			\tau(z)=\pos F,\qquad \pos F\cap P=\R_+z.
		\end{equation}
		Here, we use the duality between faces of $L=K^\circ$ and touching cones of
		$K$.
		
		Let
		\[
		\widetilde w=(\alpha x,\beta y)
		\qquad\text{and}\qquad
		w=\frac{\widetilde w}{\|\widetilde w\|}\in\mathbb S^5.
		\]
		Equations \eqref{eq:touching-product-diagonal},
		\eqref{eq:z-coordinates}, and \eqref{eq:touching-polar-duality} give
		\[
		\begin{split}
			\dim T(Q,w)=2, \ \ 
			\dim T(C,w)=3+\dim\tau(z)\leq5, \ \
			\dim\bigl(T(Q,w)\cap T(C,w)\bigr)=1.
		\end{split}
		\]
		Therefore, $w$ is $(Q[4],C)$-extreme, and $z\in\partial L$ satisfies
		$h_K(z)=1$. Thus, by homogeneity we obtain
		\[
		g(\widetilde w)
		=h_K(\alpha x)+h_K(\beta y)-h_K(z)
		=\alpha+\beta-1>0.
		\]
		This contradicts \eqref{eq:support-theorem},
		\eqref{eq:mixed-area-deficit}, and equality in
		\eqref{eq:six-dimensional-monotonicity}.

		Suppose now that \(x\) and \(y\) are linearly dependent. Since each ray emanating from the origin meets \(\partial L\) in exactly one point, two distinct points of \(\partial L\) cannot be positive multiples of one another. Hence, $y=-\lambda x$ for some
		$\lambda>0$. Let $\widetilde w=(x,-x)$.  Then,
		
		\[
		T(Q,\widetilde w)=\mathbb R_+x\times\mathbb R_+(-x),
		\qquad
		T(C,\widetilde w)=\ker s,
		\]
		and
		\[
		T(Q,\widetilde w)\cap T(C,\widetilde w)
		=\mathbb R_+(x,-x)
		\]
		which is one-dimensional. Therefore, \(\widetilde w\) determines a \((Q[4],C)\)-extreme direction after normalization. On the other hand,
		\[
		g(x,-x)=h_K(x)+h_K(-x)>0,
		\]
		yielding the same contradiction. This proves the forward implication.
		
		Conversely, assume that every pair of extreme points of \(L\) lies in a common proper exposed face, and let \(w=(p,q)\in\mathbb S^5\) be \((Q[4],C)\)-extreme. If \(p=0\) or \(q=0\), then \(g(p,q)=0\). We may therefore assume that \(p,q\neq0\). The first condition in \eqref{eq:extreme-three-conditions}, together with the product formula for \(T(Q,w)\), yields
		\[
		\dim\tau(p)+\dim\tau(q)
		=\dim T(Q,w)\leq2.
		\]
		Since \(p,q\neq0\), both \(\tau(p)\) and \(\tau(q)\) have dimension at least one. The preceding inequality therefore implies
		\[
		\dim\tau(p)=\dim\tau(q)=1.
		\]
		By polar duality,
		\[
		\frac{p}{h_K(p)}
		\qquad\text{and}\qquad
		\frac{q}{h_K(q)}
		\]
		are extreme points of \(L\). Choose \(\xi\) to expose a common proper face of \(L\) containing these points, and scale \(\xi\) so that \(h_L(\xi)=1\). It follows that
		\[
		\xi\in F(K,p)\cap F(K,q).
		\]
	Therefore,
		\[
		h_K(p+q)\geq\langle\xi,p+q\rangle
		=h_K(p)+h_K(q).
		\]
		The reverse inequality follows from the subadditivity of \(h_K\). Thus,
		\[
		g(p,q)=h_K(p)+h_K(q)-h_K(p+q)=0
		\]
		for every \((Q[4],C)\)-extreme direction. By continuity, $g$ vanishes on its closure in
		\eqref{eq:support-theorem}. Equations \eqref{eq:mixed-area-deficit} and \eqref{eq:c-mixed-volume-deficit} now give \(c=9v^2\).
	\end{proof}
	
The facial condition obtained above leads naturally to the notion of weak neighborliness. A three-dimensional polytope is weakly neighborly if any two of its vertices belong to a common proper face. The classification theorem \cite[Theorem~15]{Bayer1993} states that every weakly neighborly three-dimensional polytope is either a pyramid or a triangular prism. More generally, a \emph{generalized pyramid} is the convex hull of a two-dimensional compact
convex body and one point outside its affine hull.
	
	The facial characterization obtained above, together with the classification of weakly neighborly three-dimensional polytopes, yields the following necessary geometric condition for equality.
	
	\begin{proposition}\label{prop:polar-classification}
		Assume that $K$ satisfies equality in \eqref{eq:main-inequality}. Then $K$
		is a generalized pyramid.
	\end{proposition}
	
	\begin{proof}

		By translation invariance, we may assume that \(0\in\operatorname{int}K\) and let \(L=K^\circ\). It follows from 	\eqref{eq:both-equalities} and Proposition
		\ref{prop:c-equality-facial} that every pair of extreme points of \(L\) is contained in a common proper exposed face. Let \(E\subset\operatorname{ext}L\) be finite with \(\dim\operatorname{conv}E=3\), and define
		\[
		P_E=\operatorname{conv}E.
		\]
		Every point of \(E\) is a vertex of \(P_E\). Given \(x,y\in E\), choose a proper exposed face \(F\) of \(L\) containing both points. Then \(F\cap P_E\) is an exposed face of \(P_E\) containing \(x\) and \(y\). Since every proper face of \(L\) has dimension at most two, this face is proper. Thus, \(x\) and \(y\) lie in a common facet of \(P_E\). Hence, \(P_E\) is weakly neighborly. By the classification theorem, it is combinatorially equivalent to either a pyramid or a triangular prism.

	If \(\operatorname{ext}L\) is finite, then \(L\) is a polytope. Since \(L\) is weakly neighborly, the classification theorem shows that \(L\) is either a pyramid or a triangular prism. We therefore assume that \(\operatorname{ext}L\) is infinite. Choose a seven-element subset \(E_0\subset\operatorname{ext}L\) such that
	\[
	\dim\operatorname{conv}E_0=3.
	\]
	Since every point of \(E_0\) is a vertex of \(P_{E_0}=\operatorname{conv}E_0\), the polytope \(P_{E_0}\) has seven vertices. It cannot be a triangular prism, which has only six vertices, and must therefore be a pyramid over a hexagon. Denote its unique apex by \(a\), and let \(H\) be the plane containing the other six vertices.
	
		For any \(x\in\operatorname{ext}L\setminus E_0\), define
		\(
		P_x=\operatorname{conv}(E_0\cup{x}).
		\)
		Since all eight points are extreme points of \(L\), they are vertices of \(P_x\). Moreover, \(P_x\) is weakly neighborly. As a triangular prism has only six vertices, the classification theorem implies that \(P_x\) is a pyramid. Let \(H_x\) denote its base plane, which contains seven of its eight vertices.
		
		We claim that \(H_x=H\). Otherwise, \(H_x\cap H\) would be a line and could contain at most two of the six base vertices of \(P_{E_0}\), since no three distinct extreme points of \(L\) are collinear. The only remaining vertices are \(a\) and \(x\), so \(H_x\) could contain at most four vertices, contradicting the fact that it contains seven. Hence \(H_x=H\). Since \(a\notin H\), it follows that \(x\in H\).

		Let $B=L\cap H$.  Since $\ext L\subset B\cup\{a\}$, the
		finite-dimensional Krein--Milman theorem gives
		\[
		L=\overline{\conv}(\ext L)
		\subset\conv(B\cup\{a\})\subset L.
		\]
		Here, $B$ is compact. Thus, $\conv(B\cup\{a\})$ is closed and both inclusions
		are equalities.
		Moreover, $B$ is a two-dimensional compact convex body in $H$. Hence, $L$ is a
		generalized pyramid.  We have shown that $L$ is either a generalized pyramid
		or a triangular prism.
		
		Polarity preserves generalized pyramids.  Indeed, after a linear
		normalization we may write
		\[
		L=\conv\bigl((B\times\{-1\})\cup\{(0,t)\}\bigr),
		\qquad 0\in\operatorname{int}B\subset\R^2,\quad t>0.
		\]
		For $(y,s)\in\R^2\times\R$, the polar inequalities are
		\[
		h_B(y)-s\leq1,\qquad ts\leq1.
		\]
	Hence,
		\begin{equation}\label{eq:polar-generalized-pyramid}
			L^\circ
			=\conv\left(
			\{(0,-1)\}\cup
			\bigl((1+t^{-1})B^\circ\times\{t^{-1}\}\bigr)
			\right).
		\end{equation}
		Thus, $K=L^\circ$ is either a generalized pyramid or, by face-lattice
		duality, a triangular bipyramid.

		It remains to exclude the possibility that \(K\) is a triangular bipyramid. Let \(a\) and \(b\) be its apices, and let \(p_1,p_2,p_3\) be its equatorial vertices. The three facets
		\[
		\conv\{a,p_1,p_2\},\qquad
		\conv\{b,p_2,p_3\},\qquad
		\conv\{a,p_3,p_1\}
		\]
		have no common point. Since equality in \eqref{eq:main-inequality} implies \(b=27v^2\), Proposition \ref{prop:polytope-b-equality} implies that every three distinct facets of \(K\) to have a nonempty intersection. This contradiction excludes the case of a triangular bipyramid. Hence, \(K\) must be a generalized pyramid.
	\end{proof}
	
 The following proposition expresses the projection volume of a generalized pyramid in terms of its base and height.
	
	\begin{proposition}\label{prop:pyramid-formula}
		Let
		\begin{equation}\label{eq:normalized-pyramid}
			K=\conv\bigl((B\times\{0\})\cup\{(0,h)\}\bigr),
			\qquad h>0,\qquad A=V_2(B)>0.
		\end{equation}
		Then
		\begin{equation}\label{eq:pyramid-projection-volume}
			V_3(\Pi K)
			=\frac{Ah^2}{4}\bigl(V_2(B-B)+2A\bigr),
		\end{equation}
		and
		\begin{equation}\label{eq:pyramid-ratio}
			\frac{V_3(\Pi K)}{V_3(K)^2}
			=\frac94\left(\frac{V_2(B-B)}{A}+2\right).
		\end{equation}
	\end{proposition}
	
	\begin{proof}
		
		We first consider the case in which \(B\) is a polygon (a two-dimensional polytope). By composing a horizontal translation with a volume-preserving shear, we may assume that \(0\in\operatorname{int}B\) without changing the representation in \eqref{eq:normalized-pyramid}. Let \(p_1,\ldots,p_m\) denote the vertices of \(B\) in counterclockwise order, with indices taken modulo \(m\). Define
		\[
		e_i=p_{i+1}-p_i,
		\qquad
		d_i=\det(p_i,p_{i+1})>0.
		\]
		Then
		\begin{equation}\label{eq:d-sum}
			\sum_{i=1}^m d_i=2A.
		\end{equation}
		Let $R(a,b)=(b,-a)$.  The outer area vectors of the facets of $K$ are
		\begin{equation*}\label{eq:area-vectors}
			q_0=(0,0,-A),\qquad
			q_i=\frac12(hRe_i,d_i),\quad 1\leq i\leq m.
		\end{equation*}
		By Cauchy's formula, $\Pi K$ is the zonotope generated by these area
		vectors.  The zonotope volume formula gives
		\begin{equation}\label{eq:zonotope-discrete}
			V_3(\Pi K)
			=\sum_{0\leq i<j<k\leq m}
			|\det(q_i,q_j,q_k)|.
		\end{equation}
		Since $B-B$ is a translate of the planar zonotope
		$\sum_i[0,e_i]$, whose area is the determinant sum in
		\eqref{eq:base-terms}, the sum of all terms involving \(q_0\) is
		\begin{equation}\label{eq:base-terms}
			\frac{Ah^2}{4}\sum_{i<j}|\det(e_i,e_j)|
			=\frac{Ah^2}{4}V_2(B-B).
		\end{equation}
		
		It remains to sum the terms involving three side facets.  Take
		\(
		w_i=2q_i=(hRe_i,d_i).
		\)
		The points $Re_i/d_i$ are the vertices of $B^\circ$ in counterclockwise order.
		Hence,
		\[
		\det(w_i,w_j,w_k)>0\ \ (i<j<k).
		\]
		Expanding the determinant in the last coordinates, the coefficient of
		$h^2d_\ell$ in the sum of these determinants is
		\begin{equation*}\label{eq:C-ell}
			C_\ell=
			\sum_{\ell<j<k}\det(e_j,e_k)
			-\sum_{i<\ell<k}\det(e_i,e_k)
			+\sum_{i<j<\ell}\det(e_i,e_j).
		\end{equation*}
		To prove that $C_\ell=2A$, we define
		\[
		E_-:=\sum_{i<\ell}e_i,\qquad
		E_+:=\sum_{i>\ell}e_i,\qquad
		S:=\sum_{i<j}\det(e_i,e_j).
		\]
		Since $E_-+e_\ell+E_+=0$, we have
		\[
		\det(E_-,E_+)=-\det(E_-,e_\ell),
		\qquad
		\det(e_\ell,E_+)=\det(E_-,e_\ell).
		\]
		Decomposing \(S\) according to the positions of the indices relative to \(\ell\), and then using the preceding two identities, we obtain
		\[
		\begin{split}
			S
			&=\sum_{i<j<\ell}\det(e_i,e_j)
			+\sum_{\ell<i<j}\det(e_i,e_j)
			-\det(E_-,E_+)\\
			&=C_\ell.
		\end{split}
		\]
		Finally, using \[ p_i=p_1+\sum_{j<i}e_j \qquad\text{and}\qquad \sum_i e_i=0, \] together with the polygonal area formula and the bilinearity of the determinant, we obtain 	\[
		2A=\sum_i\det(p_i,p_{i+1})
		=\sum_i\det(p_i,e_i)
		=\sum_{j<i}\det(e_j,e_i)=S.
		\]
		Thus, $C_\ell=2A$ for every $\ell$.  Together with
		\eqref{eq:d-sum}, this yields
		\[
		\sum_{i<j<k}\det(w_i,w_j,w_k)
		=h^2\sum_\ell C_\ell d_\ell=4A^2h^2.
		\]		
		Since $q_i=w_i/2$, the terms in \eqref{eq:zonotope-discrete} involving only side facets contribute a total of \(A^2h^2/2\). Combining this identity with \eqref{eq:base-terms} proves \eqref{eq:pyramid-projection-volume} for polygons.
	
		To extend \eqref{eq:pyramid-projection-volume} to an arbitrary planar convex body \(B\), we choose a sequence of polygons \(B_k\) converging to \(B\) in the Hausdorff metric, and let
		\[
		A_k=V_2(B_k),
		\qquad
		K_k=\operatorname{conv}\bigl((B_k\times\{0\})\cup\{(0,h)\}\bigr).
		\]
		Then \(A_k\to A\), \(K_k\to K\), \(B_k-B_k\to B-B\), and \(\Pi K_k\to\Pi K\) in the Hausdorff metric. By the continuity of volume, passing to the limit in \eqref{eq:pyramid-projection-volume} for \(K_k\) establishes the same formula for \(K\). Since
		\[
		V_3(K)=\frac{Ah}{3},
		\]
		substituting this into \eqref{eq:pyramid-projection-volume} yields \eqref{eq:pyramid-ratio}.
	\end{proof}
	
	\begin{proof}[Proof of Theorem \ref{thm:equality-case}]
		Assume that equality holds in \eqref{eq:main-inequality}. Proposition \ref{prop:polar-classification} implies that \(K\) is a generalized pyramid. By affine invariance, after applying a suitable shear parallel to the base, we may assume that \(K\) has the representation in \eqref{eq:normalized-pyramid}. Combining \eqref{eq:pyramid-ratio} with the equality in \eqref{eq:main-inequality}, we obtain
		\begin{equation}\label{3.1}
		V_2(B-B)=6V_2(B).
	\end{equation}
		By the equality characterization in the planar Rogers-Shephard difference-body inequality \cite{RogersShephard1957}, identity \eqref{3.1}
		holds if and only if \(B\) is a triangle. Therefore, the generalized pyramid \(K\) with base \(B\) is a tetrahedron.
	\end{proof}

\Needspace{10\baselineskip}
	\section{Proof of Theorem~\ref{thm:high-dimensional-counterexamples}}\label{c}
	\label{sec:transfer}
	In this section, we construct counterexamples to the conjectured simplex bound in every dimension \(n\ge 9\), thereby proving Theorem~\ref{thm:high-dimensional-counterexamples}.
	
	Let
	\begin{equation*}\label{eq:simplex-cube}
		\Delta_{m-1}=\left\{s\in\R^m:s_i\geq0,\ \sum_{i=1}^m s_i=1\right\},
		\qquad
		C_m=[-\tfrac12,\tfrac12]^m.
	\end{equation*}
	Here, \(\Delta_{m-1}\) denotes the standard \((m-1)\)-dimensional simplex
	in \(\mathbb{R}^m\), and \(C_m\) denotes the centered unit cube in
	\(\mathbb{R}^m\).
	For a linear subspace \(F\subset\mathbb R^m\), let \(V_F\) denote the intrinsic Euclidean volume on \(F\). When several arguments are given, \(V_F\) denotes the corresponding mixed volume in \(F\). We write \(P_F\) for the orthogonal projection onto \(F\). When \(\dim F=n\), we define \(\mathcal P_n\) for convex bodies in \(F\) through an isometric identification of \(F\) with \(\mathbb R^n\). Let \(\mathbf 1^\perp\) denote the orthogonal complement of the vector \(\mathbf 1=(1,\ldots,1) \in \mathbb{R}^m\). 
	
	The following proposition converts suitable section projection estimates into counterexamples to the conjectured bound.

	\begin{proposition} \label{prop:transfer}
		Let $n\geq2$ and 
		$\lambda\in\relint\Delta_{m-1}$. Suppose that $F\subset\one^\perp\subset\R^m$ is an
		$n$-dimensional subspace and
		\begin{equation*}\label{eq:def-section}
			S=(\lambda+F)\cap\Delta_{m-1}.
		\end{equation*}
		Then, there exists a full-dimensional $n$-polytope $Q\subset F$ such that
		\begin{equation}\label{eq:transfer-lower}
			\mathcal P_n(Q)\geq n^n V_F(S)V_F(P_FC_m).
		\end{equation}
		In particular, if
		\begin{equation}\label{eq:section-violation}
			V_F(S)V_F(P_FC_m)>\frac{n+1}{n!},
		\end{equation}
		then $Q$ violates the conjectured inequality \eqref{eq:brannen-general}.
	\end{proposition}
	
	\begin{proof}
	Let \(e_1,\ldots,e_m\) be the standard basis of \(\mathbb R^m\), and define
	\[
	q_i=P_Fe_i\in F,
	\qquad 1\le i\le m.
	\]
	Since \(P_F\) maps \(\mathbb R^m\) onto \(F\), the vectors \(q_1,\ldots,q_m\) span \(F\). Moreover, it follows from \(F\subset\mathbf 1^\perp\) that
	\[
	\sum_{i=1}^{m} q_i
	=
	P_F\mathbf{1}
	=
	0.
	\]
	Suppose that \(v\in F\) satisfies
	\[
	\langle v,q_i\rangle\ge 0
	\ \ \text{for every }i.
	\]
	Taking the inner product of \(v\) with \(\sum_{i=1}^m q_i=0\) shows that
	\[
	\sum_{i=1}^m\langle v,q_i\rangle=0.
	\]
	Hence \(\langle v,q_i\rangle=0\) for every \(i\). Since the \(q_i\) span \(F\), it follows that \(v=0\).
	After we remove the zero vectors and combining positively parallel vectors by addition, Minkowski’s existence theorem \cite{Schneider2014} yields a full-dimensional polytope \(Q\subset F\) whose facet-area vectors are the corresponding vectors \(-q_i\). Cauchy’s projection formula then gives
		\begin{equation}\label{eq:PiQ-cube}
			\Pi Q
			=\sum_{i=1}^m[-\tfrac12q_i,\tfrac12q_i]
			=P_FC_m.
		\end{equation}
		
		Translate the section \(S\) by \(-\lambda\). That is,
		\begin{equation}\label{eq:def-shifted-section}
			\widetilde S=S-\lambda
			=\{z\in F:z_i\geq-\lambda_i\text{ for every }i\}.
		\end{equation}
		Since $z\in F$, we have
		\(
		\langle q_i,z\rangle=z_i
		\). Hence,
		\begin{equation}\label{eq:support-shifted-section}
			h_{\widetilde S}(-q_i)\leq\lambda_i.
		\end{equation}
		By the mixed-volume integral representation we yield
		\begin{equation}\label{eq:mixed-at-most}
			V_F(Q[n-1],\widetilde S)
			=\frac1n\sum_{q_i\ne0}
			\lVert q_i\rVert h_{\widetilde S}\!\left(-\frac{q_i}{\lVert q_i\rVert}\right)
			=\frac1n\sum_{q_i\ne0}h_{\widetilde S}(-q_i)
			\leq\frac1n\sum_i\lambda_i=\frac1n.
		\end{equation}
  Since every
		$\lambda_i>0$, the origin lies in $\operatorname{relint}\widetilde{S}$.
		Thus, $\widetilde S$ is an $n$-dimensional convex body in $F$.
		It follows from Minkowski's first inequality that
		\begin{equation}\label{eq:minkowski-first-transfer}
			V_F(Q)^{n-1}V_F(\widetilde S)
			\leq V_F(Q[n-1],\widetilde S)^n\leq n^{-n}.
		\end{equation}
		Since $V_F(\widetilde S)=V_F(S)$, 
		equations \eqref{eq:PiQ-cube}-\eqref{eq:minkowski-first-transfer} imply that
		\[
		\frac{V_F(\Pi Q)}{V_F(Q)^{n-1}}
		\geq n^nV_F(S)V_F(P_FC_m),
		\]
		which is \eqref{eq:transfer-lower}.  The last assertion follows immediately.
	\end{proof}
	
		Proposition \ref{prop:transfer} provides a concrete geometric construction rather than merely an abstract existence argument. After zero vectors are omitted and positively parallel
		vectors are merged, the resulting area vectors $-P_Fe_i$ determine $Q$
		uniquely up to translation.  Hence, the data considered below determine a polytope with at most $m$ facets. In our applications, $m=n+2$.
	
	Let $m=n+2$. We now consider a subspace of codimension one in \(\mathbf 1^\perp\), or equivalently, of codimension two in \(\mathbb R^m\). For a nonzero vector
	\(
	a=(a_1,\ldots,a_m)\in\mathbf 1^\perp,
	\)
	define
		\begin{equation*}\label{eq:def-Fa}
		F_a=\{\one,a\}^{\perp}.
	\end{equation*}
	Then, \(\dim F_a=m-2=n\). We also define
		\begin{equation*}\label{eq:def-Da}
		D(a)=\sum_{1\leq i<j\leq m}|a_i-a_j|.
	\end{equation*}
	The normalized Cavalieri profile associated with the family of parallel sections of the simplex is defined by
\begin{equation}\label{eq:canonical-density}
	f_a(t):=\frac{(m-1)!}{\sqrt m\,\lVert a\rVert}
	\mathcal H^{m-2}\!\left(
	\Delta_{m-1}\cap\{s:\langle a,s\rangle=t\}\right),
	\ \ \ \min_i a_i<t<\max_i a_i.
\end{equation}
	The coarea formula shows that
	\[
	\int_{\min_i a_i}^{\max_i a_i}f_a(t),dt=1.
	\]
	We use this normalization as a convenient shorthand for the Euclidean volumes of these sections.
	
	\begin{proposition}
		\label{prop:codim-one-product}
		Let $\lambda\in\relint\Delta_{m-1}$ and
		$t=\langle a,\lambda\rangle$.  Then
		\begin{equation}\label{eq:codim-one-product}
			V_{F_a}\bigl((\lambda+F_a)\cap\Delta_{m-1}\bigr)
			V_{F_a}(P_{F_a}C_m)
			=\frac{D(a)f_a(t)}{(m-1)!}.
		\end{equation}
		Since $m=n+2$, the strict inequality in \eqref{eq:section-violation} holds if and only if
		\begin{equation*}\label{eq:density-target}
			D(a)f_a(t)>(m-1)^2.
		\end{equation*}
	\end{proposition}
	
	\begin{proof}
		The intrinsic $(m-1)$-volume of the standard simplex is
		\begin{equation*}\label{eq:intrinsic-simplex-volume}
			\mathcal H^{m-1}(\Delta_{m-1})=\frac{\sqrt m}{(m-1)!}.
		\end{equation*}
		The gradient of $s\mapsto\langle a,s\rangle$ on the affine hull of the
		simplex has norm $\lVert a\rVert$.  The coarea formula and the definition of
		$f_a$ therefore give
		\begin{equation}\label{eq:section-coarea}
			V_{F_a}\bigl((\lambda+F_a)\cap\Delta_{m-1}\bigr)
			=\frac{\sqrt m\,\lVert a\rVert}{(m-1)!}f_a(t).
		\end{equation}
		
		We next compute the cube projection.  Choose an $m\times(m-2)$ matrix
		$U$ whose columns form an orthonormal basis of $F_a$.  The zonotope
		determinant
		formula \cite{McMullen1984} gives
		\begin{equation}\label{eq:zonotope-minors}
			V_{F_a}(P_{F_a}C_m)
			=\sum_{|I|=m-2}|\det U_I|.
		\end{equation}
		Adding the two orthonormal vectors
		\(m^{-1/2}\mathbf{1}\) and \(\lVert a\rVert^{-1}a\)
		as columns to \(U\) produces an orthogonal matrix. For any index set \(I\) whose complement is \(I^c=\{i,j\}\), the complementary-minor identity for orthogonal matrices yields
		\begin{equation}\label{eq:complementary-minor}
			|\det  U_I|
			=\frac{|a_i-a_j|}{\sqrt m\,\lVert a\rVert}.
		\end{equation}
		Summing \eqref{eq:complementary-minor} in
		\eqref{eq:zonotope-minors}, we obtain
		\begin{equation}\label{eq:cube-projection-D}
			V_{F_a}(P_{F_a}C_m)
			=\frac{D(a)}{\sqrt m\,\lVert a\rVert}.
		\end{equation}
		Multiplication of \eqref{eq:section-coarea} and
		\eqref{eq:cube-projection-D} proves \eqref{eq:codim-one-product}.
		
		Since $n=m-2$, we get
		\[
		\frac{n+1}{n!}
		=\frac{m-1}{(m-2)!}
		=\frac{(m-1)^2}{(m-1)!},
		\]
		which proves the last assertion,
	\end{proof}
	
We next obtain the desired strict inequality by perturbing a section from a
degenerate critical configuration. Let
\[
b=(j)_{j=-4}^{4}.
\]
For \(0<\varepsilon<1/4\), we define
	\begin{equation*}\label{eq:geometric-knots}
	a_\varepsilon=(-1,\varepsilon b,1)\in\one^\perp.
\end{equation*}
After computing the first variation, we may decrease \(\varepsilon\)
further if necessary. The condition \(\varepsilon<1/4\) ensures that
the nine middle coordinates of \(a_\varepsilon\) lie strictly between
its two outer coordinates. At \(\varepsilon=0\), these nine coordinates
all coincide at zero. For \(\varepsilon>0\), the nine middle coordinates split into nine
equally spaced levels. Let
	\begin{equation}\label{eq:base-F-lambda}
	F_\varepsilon=\{\one,a_\varepsilon\}^{\perp},
	\qquad \lambda^{(0)}=\frac1{11}\one.
\end{equation}
For \(p=(p_{-4},\ldots,p_4)\in\Delta_8\), we define
\begin{equation*}\label{eq:def-L}
	L(p)=\sum_{j=-4}^{4}jp_j.
\end{equation*}
Here, $d\bar p$ denotes normalized Euclidean volume on the simplex.
		
	\begin{lemma}
		\label{lem:cavalieri-section}
		Let
		\begin{equation*}\label{eq:def-I-epsilon}
			I(\varepsilon)=\int_{\Delta_8}
			\bigl(1+\varepsilon|L(p)|\bigr)^{-9}\,d\bar p.
		\end{equation*}
		Then,
		\begin{equation*}\label{eq:base-density}
			f_{a_\varepsilon}(0)=5I(\varepsilon).
		\end{equation*}
		Moreover, 
		\begin{equation}\label{eq:base-geometric-ratio}
			\Gamma_0(\varepsilon)
			:=\frac{
			V_{F_\varepsilon}((\lambda^{(0)}+F_\varepsilon)\cap\Delta_{10})
			V_{F_\varepsilon}(P_{F_\varepsilon}C_{11})}{10/9!}
			=\frac{D(a_\varepsilon)}{D(a_0)}I(\varepsilon).
		\end{equation}
	\end{lemma}
	
	\begin{proof}
		
		For a point of \(\Delta_{10}\), let \(w\) denote the sum of its nine
		middle barycentric coordinates. When \(w>0\), write these coordinates
		as \(wp\), where \(p\in\Delta_8\). The two outer coordinates, whose sum
		is \(1-w\), can then be parametrized as
		\[
		\frac{1-w}{2}(1-v)
		\ \ \text{and}\ \ 
		\frac{1-w}{2}(1+v),
		\ \ -1\le v\le 1.
		\]
		Note that the degenerate boundary cases \(w=0\) and \(w=1\) have measure zero and
		do not affect the subsequent integration. The normalized simplex volume element is
		\begin{equation*}\label{eq:cavalieri-volume-element}
			90w^8(1-w)\,dw\,d\bar p\,\frac{dv}{2}.
		\end{equation*}
		Indeed, because \(\Delta_8\) has dimension \(8\), scaling the middle simplex by \(w\) contributes the factor
		\(w^8\), while the outer
		one-dimensional simplex contributes the factor \(1-w\). The constant
		\(90\) is determined by
		\[
		\int_0^1 w^8(1-w)\,dw=\frac{1}{90},
		\]
		and \(dv/2\) is the normalized uniform measure on the interval
		\([-1,1]\).
		On the other hand,
		\begin{equation*}\label{eq:cavalieri-height}
			\langle a_\varepsilon,u\rangle
			=(1-w)v+\varepsilon wL(p).
		\end{equation*}
		For fixed $p$ and $w$, the corresponding central slice is nonempty when
		\[
			w\leq\frac{1}{1+\varepsilon|L(p)|}.
		\]
		Its one-dimensional coarea factor is $[2(1-w)]^{-1}$.  Cavalieri's
		principle therefore gives
		\[
			f_{a_\varepsilon}(0)
			=45\int_{\Delta_8}\int_0^{(1+\varepsilon|L(p)|)^{-1}}
			w^8\,dw\,d\bar p
			=5I(\varepsilon).
		\]
		At $\varepsilon=0$ we have $D(a_0)f_{a_0}(0)=100$.  Proposition
		\ref{prop:codim-one-product} now yields
		\eqref{eq:base-geometric-ratio}.
	\end{proof}
	
	We now establish the key strict geometric comparison needed in the
	argument. In the proof below, all integrals are taken with respect
	to normalized volume measure on the relevant cubes or simplices.
	
	\begin{lemma}
		\label{lem:cube-folding}
		Let $C=[-1/2,1/2]^8$ and $h(x)=x_1+\cdots+x_8$.  Then
		\begin{equation}\label{eq:cube-folding}
			\int_C|h(x)|\,dx<\int_Ch(x)^2\,dx.
		\end{equation}
	\end{lemma}
	
	\begin{proof}
		On the four-cube $C_4=[-1/2,1/2]^4$, take
		\[
			T(x)=x_1+\cdots+x_4,
		\]
		and let $q$ be the nondecreasing rearrangement of $T$ on the unit interval.
		Central reflection in the second factor of $C=C_4\times C_4$ shows that
		$T(x)+T(y)$ and $T(x)-T(y)$ have the same distribution.  Using equimeasurability and decomposing the unit square into the
		two regions \(s\ge t\) and \(s\le t\), we obtain
		\begin{equation}\label{eq:quantile-folding}
			\int_{C_4\times C_4}|T(x)+T(y)|\,dx\,dy
			=\int_0^1\!\int_0^1|q(s)-q(t)|\,ds\,dt
			=2\int_0^1(2s-1)q(s)\,ds.
		\end{equation}
		By symmetry under coordinate reflections, all cross terms in the
		expansion of \(T(x)^2\) integrate to zero. Hence,
		\[
			\int_0^1q(s)^2\,ds
			=\int_{C_4}T(x)^2\,dx
			=4\int_{-1/2}^{1/2}r^2\,dr
			=\int_0^1(2s-1)^2\,ds.
		\]
		Applying the Cauchy-Schwarz inequality to \eqref{eq:quantile-folding} and using the
		equality of the second moments above, it follows that
		\[
			\begin{split}
			\int_{C_4\times C_4}|T(x)+T(y)|\,dx\,dy
			&\leq2\int_0^1q(s)^2\,ds\\
			&=\int_{C_4\times C_4}(T(x)+T(y))^2\,dx\,dy.
			\end{split}
		\]
		We claim that the first inequality above is strict. Equality in the
		Cauchy--Schwarz inequality would imply that \(q(s)\) is proportional
		to \(2s-1\) almost everywhere. Since \(q\) is the quantile function of
		\(T\), this would force \(T\) to be uniformly distributed on an
		interval. This conclusion is incompatible with the geometry of the upper level sets of \(T\) near its
		maximum value \(2\). For every sufficiently small \(\delta>0\), the
		set
		\[
		T^{-1}\bigl((2-\delta,2]\bigr)
		\]
		is a vertex cap of \(C_4\) and is homothetic to a fixed four-dimensional simplex with homothety ratio
		\(\delta\). Its volume is therefore
		proportional to \(\delta^4\). By contrast, a uniformly distributed random variable would assign mass of order \(\delta\) to the interval \((2-\delta,2]\). Hence equality cannot occur, and the first inequality
		is strict. Since \(T(x)+T(y)\) is the diagonal coordinate \(h\) on
		\(C\), this proves \eqref{eq:cube-folding}.
	\end{proof}
	
	Combining the preceding strict comparison with the perturbative construction, we now establish the desired violation in dimension nine.
	
	\begin{proposition}
		\label{prop:nine-section}
		For every sufficiently small $\varepsilon>0$, the section specified by \eqref{eq:base-F-lambda} satisfies
		\begin{equation*}\label{eq:nine-section-violation}
			V_{F_\varepsilon}
			\bigl((\lambda^{(0)}+F_\varepsilon)\cap\Delta_{10}\bigr)
			V_{F_\varepsilon}(P_{F_\varepsilon}C_{11})
			>\frac{10}{9!}.
		\end{equation*}
	\end{proposition}
	
	\begin{proof}
		
		Introduce ordered variables
		\[
		0\le x_1\le\cdots\le x_8\le 1
		\]
		and parametrize \(p=(p_{-4},\ldots,p_4)\in\Delta_8\) by the successive
		gaps between these variables
		\[
		(p_{-4},\ldots,p_4)
		=
		\bigl(x_1,x_2-x_1,\ldots,x_8-x_7,1-x_8\bigr).
		\]
		In affine coordinates, this change of variables has Jacobian determinant one, and a straightforward telescoping calculation yields
	\begin{equation*}
			L(p)=4-\sum_{i=1}^8x_i.
		\end{equation*}
		Let
		\[
		\mathcal O
		=
		\bigl\{x\in[0,1]^8:0\le x_1\le\cdots\le x_8\le1\bigr\}.
		\]
		Under the gap parametrization above, the normalized volume measure
		\(d\bar p\) on \(\Delta_8\) pulls back to \(8!\,dx\) on \(\mathcal O\).
		Up to a set of measure zero, the cube \([0,1]^8\) is the disjoint union
		of the \(8!\) coordinate permutations of \(\mathcal O\). Since the functions
		\[
		x\longmapsto
		\left|4-\sum_{i=1}^{8}x_i\right|
		\ \ \text{and} \ \
		x\longmapsto
		\left(4-\sum_{i=1}^{8}x_i\right)^2
		\] are invariant under coordinate permutations, it follows that
	\[
	\begin{aligned}
		\int_{\Delta_8}|L(p)|\,d\bar p
		=
		\int_{[0,1]^8}
		\left|4-\sum_{i=1}^{8}x_i\right|\,dx \ \ \text{and} \ \
		\int_{\Delta_8}L(p)^2\,d\bar p
		=
		\int_{[0,1]^8}
		\left(4-\sum_{i=1}^{8}x_i\right)^2\,dx.
	\end{aligned}
	\]
	Under the translation \(z_i=x_i-\frac12\), we obtain
	\[
	4-\sum_{i=1}^{8}x_i=-\sum_{i=1}^{8}z_i.
	\]
Hence, it follows from Lemma~\ref{lem:cube-folding} that 
	\begin{equation}\label{eq:folded-simplex}
		\int_{\Delta_8}|L|\,d\bar p
		<\int_{\Delta_8}L^2\,d\bar p.
	\end{equation}
	
	For \(0<\varepsilon<\frac14\), the nine middle coordinates
	\(\varepsilon j\), \(-4\le j\le4\), lie strictly between the two outer
	coordinates \(-1\) and \(1\). For each middle coordinate, the sum of its
	distances to the two outer coordinates is
	\[
	(1+\varepsilon j)+(1-\varepsilon j)=2.
	\]
	Thus, the contributions involving at least one outer coordinate remain
	unchanged, and only the pairwise distances among the middle coordinates
	depend on \(\varepsilon\). Therefore,
	\[
	D(a_\varepsilon)=D(a_0)+\varepsilon D(b).
	\]
	Since the coordinates of \(b=(j)_{j=-4}^{4}\) are ordered, it follows that
	\[
	D(b)
	=
	\sum_{-4\le i<j\le4}(j-i)
	=
	2\sum_{j=-4}^{4}j^2.
	\]
For normalized volume on the standard simplex, we have
		\[
			\int_{\Delta_8}p_i^2\,d\bar p=\frac{2}{9\cdot10},
			\qquad
			\int_{\Delta_8}p_ip_k\,d\bar p=\frac{1}{9\cdot10}
			\quad(i\neq k).
		\]
		Using the standard second-moment formula for normalized volume on
		\(\Delta_8\), together with
		\[
		\sum_{j=-4}^{4}j=0,
		\]
		we obtain, upon expanding \(L^2\),
		\[
		\int_{\Delta_8}L(p)^2\,d\bar p
		=
		\left(\sum_{j=-4}^{4}j^2\right)
		\mathbin{\Big/}
		(9\cdot10).
		\]			
		Since $D(a_0)=20$, this is equivalent to
		$D(b)/D(a_0)=9\int_{\Delta_8}L^2\,d\bar p$.
		Moreover, differentiation under the integral sign gives
		\[
			I'(0+)=-9\int_{\Delta_8}|L|\,d\bar p.
		\]
		Differentiating \eqref{eq:base-geometric-ratio} from the right at \(\varepsilon=0\), and using
		\(I(0)=1\), we therefore conclude that
		\begin{equation*}\label{eq:positive-first-variation}
			\begin{aligned}
				\Gamma_0'(0+)
				&=
				\frac{D(b)}{D(a_0)}+I'(0+)\\
				&=
				9\int_{\Delta_8}
				\bigl(L(p)^2-|L(p)|\bigr)\,d\bar p
				>0,
			\end{aligned}
		\end{equation*}
		where the strict inequality follows from \eqref{eq:folded-simplex}.  Since
		$\Gamma_0(0)=1$, the desired strict inequality holds for every sufficiently
		small positive $\varepsilon$.
	\end{proof}
	
The nine-dimensional counterexample can be lifted to every higher dimension by an elementary pyramidal construction.
	
	Let \(a=(a_1,\ldots,a_m)\in\mathbb R^m\setminus\{0\}\) satisfy
	\[
	\sum_{i=1}^{m}a_i=0,
	\]
	and define
	\[
	S_a=\Delta_{m-1}\cap a^\perp.
	\]
	Adjoining a zero coordinate to \(a\) produces a corresponding section that is exactly the pyramid over \(S_a\) with apex \(e_{m+1}\). More precisely,
	identifying \(S_a\) with \(S_a\times\{0\}\subset\mathbb R^{m+1}\), we
	have
	\begin{equation*}\label{eq:pyramid-section}
		S_{(a,0)}=\operatorname{conv}(S_a,e_{m+1}).
	\end{equation*}
Since \(a\perp\mathbf 1\), the orthogonal projection of \(e_{m+1}\) onto the affine hull
\(\operatorname{aff}(S_a\times\{0\})\) is
\[
\left(\frac{1}{m}\mathbf 1,0\right).
\]
Hence the height of the pyramid \(S_{(a,0)}\), with base
\(S_a\times\{0\}\), is
\[
\operatorname{dist}\!\left(
e_{m+1},\operatorname{aff}(S_a\times\{0\})
\right)
=
\sqrt{\frac{m+1}{m}}.
\]
Since \(S_a\) has dimension \(m-2\), the pyramid \(S_{(a,0)}\) has
dimension \(m-1\). Therefore, the pyramid volume formula yields
\[
\mathcal H^{m-1}\bigl(S_{(a,0)}\bigr)
=
\frac{1}{m-1}\sqrt{\frac{m+1}{m}}\,
\mathcal H^{m-2}(S_a).
\]
Applying \eqref{eq:canonical-density} and the preceding volume identity, we obtain
	\begin{equation}\label{eq:pyramid-density-recursion}
		f_{(a,0)}(0)=\frac{m}{m-1}f_a(0).
	\end{equation}
Thus, the lifting step follows directly from the Euclidean volume formula for a pyramid.

	Choose \(\varepsilon>0\) sufficiently small such that the conclusion of Proposition~\ref{prop:nine-section} holds.  
	For $r\geq0$, extend \(a_\varepsilon\) by \(r\) zero coordinates and define
	\begin{equation}\label{eq:lifted-data}
		a_\varepsilon^{(r)}=(a_\varepsilon,0^r),
		\quad m_r=11+r,
		\quad n_r=9+r,
		\quad F_r=\{\one,a_\varepsilon^{(r)}\}^{\perp},
		\quad \lambda^{(r)}=\frac1{m_r}\one.
	\end{equation}
	Write
	\[
		D_\varepsilon=D(a_\varepsilon),
		\qquad A_\varepsilon=\sum_i|(a_\varepsilon)_i|,
	\]
	and let $D_r=D(a_\varepsilon^{(r)})$.  Repeated application of \eqref{eq:pyramid-density-recursion} yields
	\begin{equation*}\label{eq:lifted-quantities}
		f_{a_\varepsilon^{(r)}}(0)
		=\frac{10+r}{10}f_{a_\varepsilon}(0),
		\qquad D_r=D_\varepsilon+rA_\varepsilon.
	\end{equation*}
	Define
	\[
	\Gamma_r
	=\frac{D_rf_{a_\varepsilon^{(r)}}(0)}{(10+r)^2}.
	\]
	Then,
	\begin{align*}
		\Gamma_{r+1}-\Gamma_r
		&=\frac{f_{a_\varepsilon}(0)}{10}
		\frac{10A_\varepsilon-D_\varepsilon}{(10+r)(11+r)}\notag\\
		&=\frac{f_{a_\varepsilon}(0)}{10}
		\frac{4\varepsilon\sum_{j=1}^{4}j(5-j)}
		{(10+r)(11+r)}>0,
		\label{eq:ratio-increment}
	\end{align*}
Thus, the sequence $\Gamma_r$ with  $r\geq0$ is strictly increasing. Therefore, $\Gamma_r\geq\Gamma_0>1$ for every $r\geq0$.
	
	The preceding computation shows that the strict violation persists under repeated application of the lifting construction, leading to the following proposition.
	
	\begin{proposition}
		\label{prop:all-sections}
		For every $r\geq0$,
		\begin{equation*}\label{eq:all-section-violations}
			V_{F_r}\bigl((\lambda^{(r)}+F_r)\cap\Delta_{m_r-1}\bigr)
			V_{F_r}(P_{F_r}C_{m_r})
			>\frac{n_r+1}{n_r!}.
		\end{equation*}
	\end{proposition}
	
	\begin{proof}
		Since $m_r-1=10+r$ and $n_r=m_r-2$, Proposition
		\ref{prop:codim-one-product} and definition of $\Gamma_r$ imply that
		\[
			\frac{
			V_{F_r}((\lambda^{(r)}+F_r)\cap\Delta_{m_r-1})
			V_{F_r}(P_{F_r}C_{m_r})}
			{(n_r+1)/n_r!}
			=\Gamma_r>1.
		\]
	\end{proof}
	
	\begin{proof}[Proof of Theorem~\ref{thm:high-dimensional-counterexamples}]
		Fix \(n\geq 9\) and let \(r=n-9\). Applying Proposition \ref{prop:transfer} to the data defined in \eqref{eq:lifted-data}, we obtain a full-dimensional polytope
		\[
		Q_n\subset F_r\cong\mathbb{R}^n
		\]
		with at most \(m_r=n+2\) facets. Combining the estimate in Proposition \ref{prop:transfer} with Proposition \ref{prop:all-sections}, we have
		\[
		\begin{aligned}
			\mathcal{P}_n(Q_n)
			&\geq
			n^n
			V_{F_r}\bigl((\lambda^{(r)}+F_r)\cap\Delta_{n+1}\bigr)
			V_{F_r}\bigl(P_{F_r}C_{n+2}\bigr) \\
			&>
			\frac{(n+1)n^n}{n!}.
		\end{aligned}
		\]
		Thus, \(Q_n\) satisfies \eqref{eq:main-lower}, completing the proof.
	\end{proof}

\end{document}